\documentclass[12pt, leqno]{amsart}
\usepackage{amssymb}
\usepackage[dvips]{color}
\usepackage{amsmath}
\usepackage{amsthm}
\usepackage[a4paper, margin=4cm]{geometry}
\usepackage{amsfonts}
\usepackage[all]{xypic} 
\usepackage{bbm}
\usepackage{xcolor}

\usepackage[colorlinks=true,linkcolor=blue,citecolor=blue,pdfpagelabels=false]{hyperref}
\usepackage{cleveref}
\usepackage{url}

\allowdisplaybreaks[4] 
\newtheoremstyle{myremark}     {10pt}{10pt}{}{}{\bfseries}{.}{.5em}{}

\newtheorem{thm}{Theorem}[section]
\newtheorem{cor}[thm]{Corollary}
\newtheorem{lem}[thm]{Lemma}
\newtheorem{prop}[thm]{Proposition}
\theoremstyle{definition}
\newtheorem{defn}[thm]{Definition}

\theoremstyle{myremark}
\newtheorem{rem}[thm]{Remark}
\numberwithin{equation}{section}

\DeclareMathOperator{\RE}{Re}
\DeclareMathOperator{\IM}{Im}

\newcommand{\h}{\mathcal{H}}

\newcommand{\s}{\mathcal{S}}

\newcommand{\D}{\mathcal{D}}

\newcommand{\C}{\mathbb{C}}

\newcommand{\Z}{\mathbb{Z}}
\newcommand{\N}{\mathbb{N}}

\newcommand{\CP}{\mathcal{P}}
\newcommand{\BH}{\mathcal{B}(\mathcal{H})}

\newcommand{\EssD}{\mathcal{D}}

\newcommand{\abs}[1]{\left\vert#1\right\vert}
\newcommand{\set}[1]{\left\{#1\right\}}
\newcommand{\seq}[1]{\left\lbrace #1 \right\rbrace}
\newcommand{\norm}[1]{\left\Vert#1\right\Vert}

\newcommand\inner[2]{\left\langle #1, #2 \right\rangle}

\begin{document}
	
	\title[Subsequential ergodic theorems]{ Weighted Subsequential ergodic theorems  on Orlicz spaces}
	
	\author[P. Bikram]{Panchugopal Bikram}
	\address{School of Mathematical Sciences,
		National Institute of Science Education and Research,  Bhubaneswar, An OCC of Homi Bhabha National Institute,  Jatni- 752050, India}
	\email{bikram@niser.ac.in}
	
	\author[D. Saha]{Diptesh Saha}
	\address{School of Mathematical Sciences,
		National Institute of Science Education and Research,  Bhubaneswar, An OCC of Homi Bhabha National Institute,  Jatni- 752050, India}
	\email{diptesh.saha@niser.ac.in}

	\keywords{maximal ergodic inequality, individual ergodic theorems, Besicovitch weights, vector valued weights, non-commutative Orlicz spaces}
		\subjclass[2010]{ Primary: 46L55, 47A35 ; 
				Secondary: 46L51,  46L52 }
	
	\date{\today}
	

	\begin{abstract}
			For a semifinite von Neumann algebra $M$, individual convergence of subsequential, $\mathcal{Z}(M)$ (center of $M$) valued weighted ergodic averages are studied in non commutative Orlicz spaces. In the process, we also derive a maximal ergodic inequality corresponding to such averages in noncommutative $L^p~ (1 \leq p < \infty)$ spaces using the weak $(1,1)$ inequality obtained by Yeadon.
	\end{abstract}
	
	\maketitle 
	
	\section{Introduction}
	
	The connection between ergodic theory and von Neumann algebra dates back to the very inception of theory of operator algebra. The study of pointwise ergodic theorems plays a center role in classical ergodic theory and has a very deep connection with statistical physics as well. However, the study of analogous ergodic theorems in the non commutative settings originated only in the pioneering work of Lance \cite{Lance1976} in 1976. After that the theory flourished and many authors extended the results of Lance to various directions. We refer here to \cite{Conze1978}, \cite{Kuemmerer1978}, \cite{Jajte:1985vd} and the references therein.
	
	Yeadon \cite{Yeadon1977} first studied the ergodic theorems in the predual of a semifinite von Neumann algebra. He proved a maximal ergodic theorem in noncommutative $L^1$ space, which still appears frequently in modern proofs of noncommutative ergodic theorems . The corresponding maximal ergodic theorem is extended to noncommutative $L^p~ (1 < p < \infty)$ space in the celebrated work \cite{Junge2007}. Also as a consequence the analogous individual ergodic theorems are proved in the same article.
	
	On the other hand an alternative approach solely based on Yeadon's weak $(1,1)$ inequality was opted by various authors to prove various individual ergodic theorems on non commutative $L^p$ spaces. In \cite{Litvinov:2012wk}, the author introduced the notion of noncommutative uniform continuity and bilateral uniform continuity in measure at zero and provided an alternative proof of the individual ergodic theorems from \cite{Junge2007}. Several attempts has been made since then to improve the results. One natural generalisation is towards the proof of subsequential ergodic theorems. In \cite{Litvinov:2001wu}, first attempt was made to prove an individual ergodic theorem along the so called uniform sequence in the von Neumann algebra setting. Simultaneously weighted ergodic theorems also became an interesting area of research. In \cite{chilin_few_2005}, the authors studied the convergence of standard ergodic averages for actions of free groups and also for the weighted averages. Several other related works are available in the literature. The reader may look into [\cite{Anantharaman:2010ts}, \cite{Hong:2020wq}, \cite{Hong:2021vm}, \cite{Panchugopal:2022ub}] and the references therein.
	
	Another extension of these results which has been studied extensively is in the realm of symmetric spaces, in particular, the Orlicz spaces. It is known that the class of Orlicz spaces is significantly wider than the class of $L^p$ spaces. The first account of study of individual ergodic theorems in the case of noncommutative Orlicz spaces is found in \cite{Chilin:2017uj}. In \cite{Chilin:2015wo}, ergodic theorems for weighted averages is studied in fully symmetric spaces.
	
	In this article we study various ergodic theorems associated with (vector valued) weighted ergodic averages along some special subsequences in noncommutative Orlicz spaces. Before this, ergodic averages with vector valued weights has been studied in \cite{comez_ergodic_2013}. Very recently, in \cite{OBrien:2021tw}, the author studied convergence of (scalar) weighted ergodic averages along subsequences in noncommutative $L^p~ (1 \leq p < \infty)$ spaces.
	
	Our aim in this article is to establish an individual ergodic theorem for positive Dunford- Schwartz operator (see Definition \ref{Dunford}) with von Neumann algebra valued Besicovitch weighted (see Definition \ref{M valued avg}) ergodic averages along subsequence of density one in Orlicz spaces (see Theorem \ref{main thm}). Our proof essentially based upon the notion of bilateral uniform continuity in measure for normed linear spaces. 
	
	Now we describe the layout of this article. In \S 2, we collect all the  materials which are essential for this article. In particular, we recall some basic facts about von Neumann algebras $M$ with faithful normal semifinite trace $\tau$ and space of $\tau$- measurable operators. We also discuss a few topologies on this space. After that, we recollect the definition of non commutative Orlicz spaces and some of its properties which are essential for this article. We also define Dunford Schwartz operators and bilaterally uniformly equicontinuity in measure (b.u.e.m.) at zero of sequences and end this section with the recollection of few important theorems regarding this. \S3 begins with the appropriate definition of subsequential weighted ergodic averages. Then we prove a suitable form of maximal ergodic inequality and use it to prove that sequence of averages under study is b.u.e.m. at zero, which essentially helps us to obtain a convergence result in $L^1 \cap M$. Finally our main result is achieved.

	\section{Preliminaries}
	
	Throughout this article we assume that $M$ is a semifinite von Neumann algebra with faithful, normal, semifinite (f.n.s.) trace $\tau$ represented on a separable Hilbert space $\h$. Let $\CP(M)$ (resp. $\CP_0(M)$) denotes the collection of all (resp. non-zero) projections  in the von Neumann algebra $M$. For each $e \in \CP(M)$ we assign $e^\perp$ for the projection $1-e$, where $1$ denotes the identity element of $M$.
	
	Let $\BH$ denotes the space of all bounded operators of the Hilbert space $\h$. A closed densely defined operator $x: \EssD_x \subseteq \h \to \h$ is called affiliated to a $M$ if $y'x \subseteq xy'$ for all $y' \in M'$, where $M'$ denotes the commutant of  $M$ which is a von Neumann algebra by its own right. Equivalently, one can define $x$ to be affiliated to $M$ if $u'x=xu'$ holds for all unitary $u'$ in $M'$. When $x$ is affiliated to $M$, it is denoted by $x \eta M$.  The center of the von Neumann algebra $M$ is defined by $M \cap M'$ and it is denoted by $\mathcal{Z}(M)$.
	
	Now we recall that for two positive, self-adjoint operators $x,y$ defined on $\h$, $x \leq y$ is defined as: $\D_y \subseteq \D_x$ and $\norm{x^{1/2}\xi}^2 \leq \norm{y^{1/2}\xi}^2$ for all $\xi \in \D_y$.
	
	\begin{prop}
		Let $x$ be a positive, self-adjoint operator affiliated to  $M$ and $z \in \mathcal{Z}(M)_+$ be such that $z \leq C$ for some constant $C>0$. Then $ 0 \leq zx \leq Cx$.
	\end{prop}
	
	\begin{proof}
		First observe that $\D_{zx}= \D_x \subseteq \D_{x^{1/2}}$. Also, $\D_{zx} \subseteq \D_{(zx)^{1/2}}$. Let $\xi \in \D_x$. Then
		\begin{align*}
			\norm{(zx)^{1/2} \xi}^2 = \inner{zx \xi}{\xi} &= \inner{xz \xi}{ \xi} ~~ \text{( since $zx \subset xz$)}\\
			&= \inner{x^{1/2}z \xi}{x^{1/2} \xi} ~~ \text{( since $\xi \in \D_{x^{1/2}}$)}\\
			&= \inner{z x^{1/2} \xi}{x^{1/2} \xi} ~~ \text{( since $x \eta M$)}\\
			&\leq C \norm{x^{1/2} \xi}^{2}.
		\end{align*}
	\end{proof}
	
	A closed, densely defined operator $x$ affiliated to $M$ is said to be $\tau$-measurable if for every $\epsilon >0$ there is a projection $e$ in $M$ such that $e \h \subseteq \EssD_x$ and $\tau(e^\perp)< \epsilon$. The set of all $\tau$-measurable operators associated to $M$ is denoted by $L^0(M, \tau)$ or simply $L^0$. For all $\epsilon, \delta >0$, let us define the following neighborhoods of zero.
	\begin{align*} 
		\mathcal{N}(\epsilon, \delta):= 
		\{ x \in L^0: \exists ~ e \in \CP(M) \text{ such that }  \norm{xe} \leq \epsilon \text{ and } \tau(e^\perp) \leq \delta \},
		\text{ and }\\
		\mathcal{N'}(\epsilon, \delta):= 
		\{ x \in L^0: \exists ~ e \in \CP(M) \text{ such that }  \norm{exe} \leq \epsilon \text{ and } \tau(e^\perp) \leq \delta \}.
	\end{align*}
	
	\noindent It is established in \cite[Theorem 2.2]{chilin_few_2005} that the families $\{\mathcal{N}(\epsilon, \delta) : \epsilon>0, \delta>0 \}$ and $\{\mathcal{N'}(\epsilon, \delta) : \epsilon>0, \delta>0 \}$ generate same topology on $L^0$, and it is termed as measure topology in the literature. It is also well-known that $L^0$ becomes a complete, metrizable topological $*$-algebra with respect to the measure topology containing $M$ as a dense subspace [see \cite[Theorem 4.12]{Hiai:2021vl}].
	
	In this article, we also deal with so called almost uniform (a.u.) and bilateral almost uniform (b.a.u) convergence of sequences in $L^0$. We describe it in the following definition.
	
	\begin{defn}
		A sequence of operators $\{x_n\}_{n \in \N} \subset L^0$ converges a.u. (resp. b.a.u.) to $x \in L^0$ if for all $\delta > 0$ there exists a projection $e \in M$ such that
		\begin{align*}
			\tau(e^\perp) < \delta \text{ and } \lim_{n \to \infty} \norm{(x_n-x) e} = 0 ~~(\text{resp. } \tau(e^\perp) < \delta \text{ and } \lim_{n \to \infty} \norm{e (x_n-x) e} =0).
		\end{align*}
	\end{defn}
	
	Now we recall the following useful lemma from \cite[Lemma 3]{Litvinov:2001wu}.
	
	\begin{lem}\label{app by covg seq}
		If a sequence $\{a_n\}_{n \in \N} \subset M$ is such that for every $\epsilon >0$ there is a b.a.u. (a.u. ) convergent sequence $\{b_n\}_{n \in \N} \subset M$ and a positive integer $N_0$ satisfying	$\norm{a_n - b_n} < \epsilon$ for all $n \geq N_0$, then $\{a_n\}_{n \in \N}$ converges b.a.u. (a.u.).
	\end{lem}

	Next we provide a brief description of noncommutative Orlicz spaces. We follow \cite{Pisier:2003wz} as our main references.
	
	\subsection{Noncommutative Orlicz spaces.} Let $M$ be a von Neumann algebra equipped with a f.n.s trace $\tau$ as mentioned above. The trace $\tau$ is extended to the positive cone $L^0_+$ of $L^0$ as follows. Suppose  $x \in L^0_+$ with the spectral decomposition $x= \int_{0}^{\infty} \lambda de_{\lambda}$. Then $\tau(x)$ is defined by
	\begin{align*}
		\tau(x):= \int_{0}^{\infty} \lambda d\tau(e_{\lambda}).
	\end{align*}
	For $0< p \leq \infty$, the noncommutative $L^p$-space associated to $(M, \tau)$ is defined as
	\begin{align*}
		L^p(M,\tau):= 
		\begin{cases}
			\{x \in L^0 : \norm{x}:= \tau(\abs{x}^p)^{1/p} < \infty\}  & \text{ for } p \neq \infty \\
			(M, \norm{\cdot}) & \text{ for } p= \infty
		\end{cases}
	\end{align*} 
	where, $\abs{x}= (x^*x)^{1/2}$. From here onwards we will simply write $L^p$ for noncommutative $L^p$-spaces.
	
	Let $x \in L^0$. Consider the spectral decomposition $\abs{x}= \int_{0}^{\infty} s de_s$. The distribution function of $x$ is defined by 
	\begin{align*}
		(0, \infty): s \mapsto \lambda_s(x):= \tau(e_s^\perp(\abs{x})) \in [0, \infty]
	\end{align*}
	and the generalised singular number of $x$ is defined by
	\begin{align*}
		(0, \infty): t \mapsto \mu_t(x):= \inf \{s>0 : \lambda_s(x) \leq t \} \in [0, \infty].
	\end{align*}
	Note that both the functions are decreasing and continuous from right on $(0, \infty)$. Among many other properties of generalised singular number, here we recall the following ones which will be used later.
	
	\begin{prop}\label{sin num prop}
		Let $a,b,c \in L^0$. Then 
		\begin{enumerate}
			\item[$(i)$] $\mu_t(f(\abs{a}))= f(\mu_t(a))$, $t>0$ and for any continuous increasing function $f$ on $[0, \infty)$ with $f(0) \geq 0$.
			\item[$(ii)$] $\mu_t(bac) \leq \norm{b} \norm{c} \mu_t(a)$ for all $t>0$.
			\item[$(iii)$] $\tau(f(\abs{a}))= \int_{0}^{\infty} f(\mu_t(a)) dt$ for any continuous increasing function $f$ on $[0, \infty)$ with $f(0) = 0$.
		\end{enumerate}
	\end{prop}
	
	\begin{proof}
		For the proofs we refer to \cite[Lemma 2.5 and Corollary 2.8]{Fack:1986ua}.
	\end{proof}
	
	\begin{defn}
		A convex function $\Phi: [0, \infty) \to [0, \infty)$ which is continuous at $0$ with $\Phi(0)=0$ and $\Phi(t)>0$ when $t \neq 0$ is called an Orlicz function.
	\end{defn}
	\noindent It is to be noted that the convexity of the function $\Phi$ and continuity at $0$ imply that the function is continuous on $[0, \infty)$. Moreover, it is also evident that $\Phi(\lambda t) \leq \lambda \Phi(t)$ whenever $0 \leq \lambda \leq 1$ and $t \in [0, \infty)$, which implies $\Phi(t_1)< \Phi(t_2)$ for all $0 \leq t_1 < t_2$. Hence the function $\Phi$ is increasing. The following result from \cite[Lemma 2.1]{Chilin:2017uj} is crucial.
	\begin{lem}\label{orlicz fn prop}
		Let $\Phi$ be an Orlicz function. Then for all $\delta >0$ there exists $u>0$ satisfying the condition
		\begin{align*}
			u \cdot \Phi(t) \geq t ~~ \text{ whenever } t \geq \delta.
		\end{align*}
		In particular, $\lim_{t \to \infty} \Phi(t)= \infty$.
	\end{lem}
	
	Now let $\Phi$ be an Orlicz function and consider $x \in L^0_+$ with the spectral decomposition $x= \int_{0}^{\infty} \lambda d(e_{\lambda})$. Then by means of functional calculus, we have
	\begin{align*}
		\Phi(x) = \int_{0}^{\infty} \Phi(\lambda) de_{\lambda}.
	\end{align*}
	The noncommutative Orlicz space associated to $(M, \tau)$ for the Orlicz function $\Phi$ is defined as
	\begin{align*}
		L^\Phi= L^\Phi(M,\tau) := \Big\{x \in L^0: \tau \Big(\Phi \Big(\frac{\abs{x}}{\lambda} \Big) \Big) < \infty \text{ for some } \lambda>0 \Big\}.
	\end{align*}
	The space $L^\Phi$ is equipped with the norm (called Luxemburg norm)
	\begin{align*}
		\norm{x}:= \inf \Big\{ \lambda>0 : \tau \Big(\Phi \Big(\frac{\abs{x}}{\lambda} \Big) \Big) \leq 1 \Big\}, ~~ x \in L^\Phi.
	\end{align*}
	It follows from \cite[Proposition 2.5]{Kunze:1990vc} that $L^\Phi$ equipped with the norm defined above is a Banach space. We now prove the following result.

	\begin{prop}\label{orlicz norm prop}
		Suppose  $x \in L^\Phi$, then
		\begin{enumerate}
			\item[$(i)$]  if  $a,b \in M$, then  $axb \in L^\Phi$. Moreover, $\norm{axb}_\Phi \leq \norm{a} \norm{b} \norm{x}_\Phi$ and 
			\item[$(ii)$] if  $\norm{x} \leq 1$, then $\tau(\Phi(\abs{x})) \leq \norm{x}$.
		\end{enumerate}
	\end{prop}
	
	\begin{proof}
		\emph{$(i)$}  Let $\lambda>0$ and observe that
		\begin{align}\label{tau ineq}
			\begin{split}
				\tau \Big(\Phi \Big(\frac{\abs{axb}}{\norm{a} \norm{b} \lambda} \Big) \Big)
				&= \int_{0}^{\infty} \Phi \Big(\mu_t \Big(\frac{axb}{\norm{a} \norm{b}\lambda} \Big) \Big) dt ~~[\text{ by $(iii)$ of Proposition \ref{sin num prop}} ] \\
				& \leq \int_{0}^{\infty} \Phi \Big(\mu_t \Big(\frac{x}{\lambda} \Big) \Big) dt ~~[\text{ by $(ii)$ of Proposition \ref{sin num prop}}] \\
				&= \tau \Big(\Phi \Big(\frac{\abs{x}}{\lambda} \Big) \Big) ~~[\text{ by $(iii)$ of Proposition \ref{sin num prop} }].
			\end{split}
		\end{align}
		Then, note that
		\begin{align*}
			\inf \{\lambda >0 : \tau \Big(\Phi \Big(\frac{\abs{axb}}{\lambda} \Big) \Big) \leq 1\}
			&= \inf \{\norm{a} \norm{b} \lambda >0 : \tau \Big(\Phi \Big(\frac{\abs{axb}}{\norm{a} \norm{b} \lambda} \Big) \Big) \leq 1\} \\
			&= \norm{a} \norm{b} \inf \{\lambda >0 : \tau \Big(\Phi \Big(\frac{\abs{axb}}{\norm{a} \norm{b} \lambda} \Big) \Big) \leq 1\}.
		\end{align*}
		Therefore, by Eq. \ref{tau ineq} we have 
		\begin{align*}
			\norm{axb}_\Phi = \inf \{\lambda >0 : \tau \Big(\Phi \Big(\frac{\abs{axb}}{\lambda} \Big) \Big) \leq 1\} & \leq \norm{a} \norm{b} \inf \{\lambda >0 : \tau \Big(\Phi \Big(\frac{\abs{x}}{\lambda} \Big) \Big) \leq 1 \} \\
			& = \norm{a} \norm{b} \norm{x}_\Phi.
		\end{align*}
		\emph{Proof of $(ii)$};  it follows immediately from \cite[Proposition 2.2]{Chilin:2017uj}.
	\end{proof}
	
	Let us now recall that a Banach space $(E, \norm{\cdot}) \subset L^0$ is called fully symmetric if
	\begin{align*}
		x \in E, ~ y \in L^0, ~ \int_{0}^{s} \mu_t(y) dt \leq \int_{0}^{s} \mu_t(x) dt ~~  \forall ~~ s>0 ~~ \Rightarrow ~~ y \in E ~\text{ and } \norm{y} \leq \norm{x}
	\end{align*}
	and a fully symmetric space $(E, \norm{\cdot}) \subseteq L^0$ is said to have Fatou Property if 
	\begin{align*}
		x_{\alpha} \in E_+, ~~ x_{\alpha} \leq x_{\beta} ~~ \text{ for } \alpha \leq \beta  \text{ and } \sup_{\alpha} \norm{x_{\alpha}} < \infty \Rightarrow \exists ~ x= \sup_{\alpha} x_\alpha \in E \text{ and } \norm{x}= \sup_{\alpha} \norm{x_{\alpha}}.
	\end{align*}
	Now the following proposition holds true.
	
	\begin{prop}\label{fatou prop of orlicz}
		$(L^{\Phi}, \norm{\cdot})$ is a fully symmetric space with the Fatou property and an exact interpolation space for the Banach couple $(L^1, M)$.
	\end{prop}
	
	\begin{proof}
		Proof follows from \cite[Corollary 2.2]{Chilin:2017uj}.
	\end{proof}
	
	As a consequence we remark the following.
	\begin{rem}\label{unit ball closed}
		It follows from \cite[Theorem 4.1]{Dodds:2005ty} and Proposition \ref{fatou prop of orlicz} that unit ball of $(L^{\Phi}, \norm{\cdot})$ is closed under measure topology.
	\end{rem}
	
	\begin{defn}
		An Orlicz function $\Phi$ is said to satisfy $\Delta_2$ condition if there exists $d>0$ such that
		\begin{align*}
			\Phi(2t) \leq d \Phi(t) \text{ for all } t \geq 0.
		\end{align*}
	\end{defn}
	\noindent Observe that for every $1 \leq p < \infty$, $\Phi(u)= \frac{u^p}{p}$, $u \geq 0$ is an Orlicz function which satisfy the $ \Delta_2$ condition. Also, in this case $L^\Phi=L^p$ for all $1 \leq p < \infty$.
	
	\begin{prop}\label{dense subsp of Orlicz sp}
		Let $\Phi$ be an Orlicz function satisfying $\Delta_2$ condition. Then the linear subspace $L^1 \cap M$ is dense in $(L^{\Phi}, \norm{\cdot})$.
	\end{prop}
	
	\begin{proof}
		For the proof we refer to \cite[Proposition 2.3]{Chilin:2017uj}.
	\end{proof}
	
	\begin{defn}\label{Dunford}
		A linear map $T: L^1 + M \to L^1 + M$ is called Dunford-Schwartz operator if it contracts both $L^1$ and $M$, i.e,
		\begin{align*}
			\norm{Tx}_\infty \leq \norm{x}_\infty ~\forall~ x \in M \text{ and } \norm{Tx}_1 \leq \norm{x}_1 ~\forall~ x \in L^1.
		\end{align*}
		If in addition $T(x) \geq 0$ for all $x \geq 0$ then we call $T$ is a positive Dunford-Schwartz operator. We write $T \in DS$ (resp. $T \in DS^+$) to denote $T$ is a Dunford-Schwartz operator (resp. positive Dunford-Schwartz operator).
	\end{defn}
	
	Let $T \in DS$. Then observe that for an Orlicz function $\Phi$ the space $L^\Phi$ is an exact interpolation space for the Banach couple $(L^1, M)$ (by Proposition \ref{fatou prop of orlicz}). Therefore we have
	\begin{align*}
		T(L^\Phi) \subseteq L^\Phi \text{ and } \norm{T} \leq 1.
	\end{align*}
	
	\begin{defn}
		Let $(X, \norm{\cdot})$ be a normed linear space and $Y \subseteq X$ be such that the zero of $X$ is a limit point of $Y$. A family of maps $A_\alpha: X \to L^0$, $\alpha \in I$, is called uniformly equicontinuous in measure (u.e.m) [ bilaterally uniformly equicontinuous in measure (b.u.e.m)] at zero on $Y$ if for all $\epsilon, \delta>0$,  there exists $\gamma>0$  such that for all $x \in Y$ with $\norm{x} < \gamma$ there exists $e \in \CP(M)$ such that
		\begin{align*}
			\tau(e^{\perp}) < \epsilon \text{ and } \sup_{\alpha \in I} \norm{A_\alpha(x) e}_{\infty} <\delta ~~(\text{respectively, } \sup_{\alpha \in I} \norm{e A_\alpha(x) e}_{\infty} <\delta).
		\end{align*}
	\end{defn}
	
	Now we recall the following significant result from \cite[Theorem 2.1]{Litvinov:2012wk} which will play an important role in our studies.
	\begin{thm}\label{bau closed set 1}
		Let $(X, \norm{\cdot})$ be a Banach space and $A_n: X \to L^0$ be a sequence of additive maps. If the sequence $\seq{A_n}_{n \in \N}$ is b.u.e.m (u.e.m.) at zero on $X$, then the set
		\begin{align*}
			\set{x \in X : \seq{A_n(x)} \text{ converges } \text{b.a.u (a.u.)}}
		\end{align*}
		is closed in $X$.
	\end{thm}
	
	We end this section with a brief introduction to density and lower density of a sequence of natural numbers.
	\begin{defn}
		A sequence $\mathbf{k}:= \{k_j\}_{j \in \N}$ of natural numbers is said to have density (resp, lower density) $d$ if
		\begin{align*}
			\lim_{n \to \infty} \frac{\abs{\{0,1, \ldots, n\} \cap \mathbf{k}}}{n+1}= d ~~ (\text{resp, } \liminf_{n \to \infty} \frac{\abs{\{0,1, \ldots, n\} \cap \mathbf{k}}}{n+1}= d).
		\end{align*}
	\end{defn}
	
	\begin{rem}\label{rem about lower density}
		We remark that if a sequence $\mathbf{k}$ has density $d$, then $\lim_{n \to \infty} \frac{k_n}{n}= \frac{1}{d}$. Moreover, we recall from \cite[Lemma 40]{Rosenblatt:1994vk} that a sequence $\mathbf{k}$ has lower density $d$ if and only if $\sup_{n \in \N} \frac{k_n}{n} < \infty$.
	\end{rem}

	\section{Convergence along sequence of density one}
	
	Throughout this section $M$ is assumed to be a semifinite von Neumann algebra with f.n.s trace $\tau$ and $T \in DS^+$. In this section, we will study the convergence of ergodic averages with $M$-valued Besicovitch weights (see Definition \ref{M valued avg} and Definition \ref{besicovitch seq}) along sequence of density one. In particular, we will prove the b.a.u. convergence of sequences of such averages in the spaces $L^\Phi$ for some Orlicz function $\Phi$. Convergence of usual vector valued weighted averages in norm and b.a.u. topology has already been studied in \cite{comez_ergodic_2013}. In this section, we also extend some of these results. We begin with few definitions of ergodic averages.
	
	\begin{defn}\label{M valued avg}
		Let $T \in DS^+$.
		For $\{b_j\}_{j \in \N} \subset M$ and $\{d_j\}_{j \in \N} \subset M$ and any sequence $\mathbf{k}:= \{k_j\}_{j \in \N}$ of natural numbers, define
		\begin{align*}
			& A_n(\{b_j\}, \{d_j\}, x):= \frac{1}{n} \sum_{j=0}^{n-1} T^j (b_j x d_j),  \hspace{2em}	A_n(\{b_j\}, x):= \frac{1}{n} \sum_{j=0}^{n-1} T^j (b_j x);
			\\
			\text{and } 
			& A_n^{\mathbf{k}}(\{b_j\}, \{d_j\}, x)):= \frac{1}{n} \sum_{j=0}^{n-1}  T^{k_j} (b_{k_j} x d_{k_j}), \hspace{2em} 	A_n^{\mathbf{k}}(\{b_j\}, x):= \frac{1}{n} \sum_{j=0}^{n-1} T^{k_j} (b_{k_j} x)
		\end{align*}
		for all $n \in \N$ and $x \in L^1 + M$. 
	\end{defn}
	
	Here we observe that when the sequence $\{b_j\}_{j \in \N}$ consists of only scalars $\beta: =\{\beta_j\}_{j \in \N}$ and the set $\{d_j\}_{j \in \N}$ consists of only identity of $M$, then the averages mentioned above will be denoted by $A_n^\beta(x)$ and $A_n^{\beta, \mathbf{k}}(x)$ respectively for $x \in L^1 + M$. Convergence of such averages are studied in \cite{OBrien:2021tw}.
	
	Let us now recall the following maximal ergodic theorem from \cite{Yeadon1977}. This result is crucial in obtaining a maximal ergodic inequality in the form required for our purpose.
	
	\begin{thm}\label{maximal erg thm yeadon}
		Let $T \in DS^+$. Then for all $x \in L^1_+$ and $\epsilon>0$ there exists $e \in \CP(M)$ such that 
		\begin{align*}
			\tau(e^\perp) \leq \frac{\norm{x}_1}{\epsilon} ~~ \text{and } \sup_{n \in \N} \norm{e A_n(\{1\}, x) e} \leq \epsilon.
		\end{align*}	
	\end{thm}
	Although the following lemma is a part of the proof of Theorem 2.1 in \cite{Chilin:2015wo}, we include the proof here for the sake of completeness.
	
	\begin{lem}\label{maximal erg thm}
		Let $1 \leq p < \infty$, $x \in L^p_+$ and $\epsilon >0$. Then there exists $e \in \CP(M)$ such that
		\begin{align*}
			\tau(e^{\perp}) \leq \Big(\frac{\norm{x}_p}{\epsilon}\Big)^p ~~ \text{ and }~~ \sup_{n \in \N} \norm{e A_n(\{1\}, x) e}_{\infty} \leq 2 \epsilon
		\end{align*}
	\end{lem}
	
	\begin{proof}
		Consider the spectral decomposition of $x= \int_{0}^{\infty} \lambda de_\lambda$. Note that since $\lambda \geq \epsilon \Rightarrow \lambda \leq \epsilon^{1-p} \lambda^p$, we have
		\begin{align*}
			\int_{\epsilon}^{\infty} \lambda de_\lambda \leq \epsilon^{1-p} \int_{\epsilon}^{\infty} \lambda^p de_\lambda \leq \epsilon^{1-p} x^p.
		\end{align*}
		Therefore, we obtain
		\begin{align*}
			x= \int_{0}^{\epsilon} \lambda de_\lambda + \int_{\epsilon}^{\infty} \lambda de_\lambda \leq x_\epsilon + \epsilon^{1-p} x^p,
		\end{align*}
		where $x_\epsilon= \int_{0}^{\epsilon} \lambda de_\lambda$. Now since $x^p \in L^1_+$, it follows from Theorem \ref{maximal erg thm yeadon} that there exist $e \in \CP(M)$ such that 
		\begin{align*}
			\tau(e^\perp) \leq \frac{\norm{x^p}_1}{\epsilon^p}= \Big(\frac{\norm{x}_p}{\epsilon} \Big)^p ~~ \text{and } \sup_{n \in \N} \norm{e A_n(\{1\}, x^p) e} \leq \epsilon^p.
		\end{align*}
		Consequently, for all $n \in \N$ we have 
		\begin{align*}
			0 \leq e A_n(\{1\}, x) e \leq e A_n(\{1\}, x_\epsilon) e + \epsilon^{1-p} e A_n(\{1\}, x^p) e.
		\end{align*}
		Since $x_\epsilon \in M$ and $\norm{T(x_\epsilon)}_\infty \leq \norm{x_\epsilon}_\infty \leq \epsilon$, we conclude that
		\begin{align*}
			\sup_{n \in \N} \norm{e A_n(\{1\}, x) e}_{\infty} \leq 2 \epsilon.
		\end{align*}
	\end{proof}
	
	Now the following result holds.
	
	\begin{thm}\label{maximal erg thm-vec weight}
		Let $\{b_j\}_{j \in \N}$ be a bounded sequence in $\mathcal{Z}(M)$ and $x \in L^p ~~(1 \leq p < \infty)$. Then for all $\epsilon>0$ there exists $e \in \CP(M)$ such that
		\begin{align*}
			\tau(e^{\perp}) \leq 4 \Big(\frac{\norm{x}_p}{\epsilon}\Big)^p ~~ \text{ and }~~ \sup_{n \in \N} \norm{e A_n(\{b_j\}, x) e}_{\infty} \leq 48 C \epsilon,
		\end{align*}
		where $C= \sup_{j \in \N} \norm{b_j}_{\infty}$.
	\end{thm}
	
	\begin{proof}
		First consider $x \in L^p_+$ and observe that if $b_j= 1$ for all $j \in \N$, then it follows from Lemma \ref{maximal erg thm} that for all $\epsilon>0$ there exists $e \in \CP(M)$ such that
		\begin{align}\label{max erg thm eq1}
			\tau(e^{\perp}) \leq  \Big(\frac{\norm{x}_p}{\epsilon}\Big)^p ~~ \text{ and }~~ \sup_{n \in \N} \norm{e A_n(\{1\}, x) e}_{\infty} \leq 2 \epsilon.
		\end{align}
		
		Now consider $\{b_j\}_{j \in \N}$ to be  a bounded sequence in $\mathcal{Z}(M)$ with $\norm{b_j}_{\infty} \leq C$ for all $j \in \N$. Then we have $0 \leq \RE(b_j) + C \leq 2C$ and similarly $0 \leq \IM(b_j) + C \leq 2C$ for all $j \in \N$. Therefore, we must have for all $j \in \N$
		\begin{align*}
			0 \leq (\RE(b_j) + C)x \leq 2Cx ~~ \text{ and } 0 \leq (\IM(b_j) + C)x \leq 2Cx.
		\end{align*}
		Also, for all $j \in \N$, we have 
		\begin{align*}
			T^j(b_j x) = T^j((\RE(b_j) + C)x) + i T^j((\IM(b_j) + C)x) - (1+i)CT^j(x).
		\end{align*}
		Then Eq \ref{max erg thm eq1} implies that for all $\epsilon>0$ there exists $e \in \CP(M)$ such that
		\begin{align}\label{max erg thm eq2}
			\tau(e^{\perp}) \leq  \Big(\frac{\norm{x}_p}{\epsilon}\Big)^p ~~ \text{ and }~~ \sup_{n \in \N} \norm{e A_n(\{b_j\}, x) e}_{\infty} \leq 6C \sup_{n \in \N} \norm{e A_n(\{1 \}, x) e}_{\infty}   \leq 12C \epsilon.
		\end{align}
		For $x \in L^p$, write $x = (x_1 - x_2)+ i (x_3 - x_4)$, where $x_l \in L^p_+$ and $\norm{x_l}_p \leq \norm{x}_p$ for all $l \in \{1, \ldots, 4\}$. Therefore, it follows from Eq \ref{max erg thm eq2} that there exist projections $e_l \in M$ such that 
		\begin{align*}
			\tau(e_l^{\perp}) \leq  \Big(\frac{\norm{x}_p}{\epsilon}\Big)^p ~~ \text{ and }~~ \sup_{n \in \N} \norm{e_l A_n(\{b_j\}, x) e_l}_{\infty}  \leq 12C \epsilon ~~ \text{ for all } l \in \{1, \ldots, 4\}.
		\end{align*}
		Now consider $e= \wedge_{l=1}^4 e_l$ to obtain the required result.
	\end{proof}
	
	Before we move to our next theorem we need to fix some notations. From here onwards $\mathbf{k}:= \{k_j\}_{j \in \N}$ will always denote a strictly increasing sequence of natural numbers. For any sequence $\{b_j\}_{j \in \N} \subset M$ and $n \in \N$, $A_n(\{b_j\}, x)$ recall the definition of $A_n^{\mathbf{k}}(\{b_j\}, x)$ and $	A_n^{\mathbf{k}}(\{b_j\}, x)$ from Definition \ref{M valued avg}, where $x \in L^1 + M$.
	
	\begin{thm}\label{equi cont thm 1}
		Let $\{b_j\}_{j \in \N}$ be a bounded sequence in $\mathcal{Z}(M)$. If the strictly increasing sequence $\mathbf{k}:= \{k_j\}_{j \in \N}$ of natural numbers has lower density $d>0$, then the sequences $\{A_n(\{b_j\}, \cdot)\}_{n \in \N}$ and $\{A_n^{\mathbf{k}}(\{b_j\}, \cdot)\}_{n \in \N}$ are b.u.e.m at zero on $(L^{\Phi}, \norm{\cdot}_{\Phi})$.
	\end{thm}
	
	\begin{proof}
		It is enough to to show that the sequences $\{A_n(\{b_j\}, \cdot)\}_{n \in \N}$ and $\{A_n^{\mathbf{k}}(\{b_j\}, \cdot)\}_{n \in \N}$ are b.u.e.m at zero on $(L^{\Phi}_+, \norm{\cdot}_{\Phi})$. 
		
		Now fix $\epsilon, \delta >0$. Then by Lemma \ref{orlicz fn prop}, there exists a $t>0$ such that 
		\begin{align*}
			t \cdot \Phi(\lambda) \geq \lambda ~~ \text{ for all } \lambda \geq \frac{\delta}{2C} .
		\end{align*}
		Choose $0< \gamma < \min \{1, \frac{\delta \epsilon}{4 \times 96 Ct}\}$. Let $x \in L^{\Phi}_+$ with $\norm{x}_{\Phi}< \gamma$ and let $x= \int_{0}^{\infty} \lambda d e_{\lambda}$ be its spectral decomposition. Then we can write
		\begin{align*}
			x= \int_{0}^{\frac{\delta}{2C}} \lambda d e_{\lambda} + \int_{\frac{\delta}{2C}}^{\infty} \lambda d e_{\lambda} \leq x_{\delta} + t \int_{\frac{\delta}{2C}}^{\infty} \Phi(\lambda) d e_{\lambda} \leq  x_{\delta} + t\Phi(x),
		\end{align*}
		where $x_{\delta}= \int_{0}^{\frac{\delta}{2C}} \lambda d e_{\lambda}$ and $\Phi(x)= \int_{0}^{\infty} \Phi(\lambda) d e_{\lambda}$.
		
		Observe that, $\norm{x_{\delta}} \leq \frac{\delta}{2C}$ and since $T$ is a positive Dunford-Schwarz operator we must have
		\begin{align*}
			\sup_{n \in \N} \norm{A_n(\{b_j\}, x_{\delta})} \leq \frac{C\delta}{2C}= \frac{\delta}{2},
		\end{align*}
		where, $C= \sup_{j \in \N} \norm{b_j}_{\infty}$.
		
		Also, since $\norm{x}_{\Phi}< \gamma <1$, by Proposition \ref{orlicz norm prop} we have $\norm{\Phi(x)}_1 \leq \norm{x}_{\Phi} $. Furthermore, since $\Phi(x) \in L^{1}_+$, by Theorem \ref{maximal erg thm-vec weight} we find $e \in \CP(M)$ satisfying
		\begin{align*}
			& \tau(e^{\perp})< \frac{4 \times 96Ct \norm{\Phi(x)}_1}{\delta} \leq \frac{4 \times 96Ct \norm{x}_{\Phi}}{\delta} < \epsilon \\
			~~\text{ and, }~~ \\
			& \sup_{n \in \N} \norm{e A_n(\{b_j\},\Phi(x)) e} < \frac{48 C \delta}{96Ct} = \frac{\delta}{2t}.
		\end{align*}
		Therefore,
		\begin{align*}
			\sup_{n \in \N} \norm{e A_n(\{b_j\}, x) e}
			& \leq \sup_{n \in \N} \norm{e A_n(\{b_j\}, x_{\delta}) e} + t \cdot \sup_{n \in \N} \norm{e A_n(\{b_j\}, \Phi(x)) e} \\
			& < \frac{\delta}{2} + t \cdot \frac{\delta}{2t} = \delta.
		\end{align*}
		
		Hence, the sequence $\{A_n(\{b_j\}, \cdot)\}_{n \in \N}$ is b.u.e.m at zero on $(L^{\Phi}_+, \norm{\cdot}_{\Phi})$. To show the sequence $\seq{A_n^{\mathbf{k}}(\{b_j\}, \cdot)}_{n \in \N}$ is b.u.e.m at zero on $(L^{\Phi}_+, \norm{\cdot}_{\Phi})$, we first consider the sequence $\{c_j\}_{j \in \N}$, where for all $j \in \N$, $c_j:= \chi_{\mathbf{k}} (j)$.
		
		Observe that for all $n \in \N$,
		\begin{align}
			A_n^{\mathbf{k}}(\{ b_j\}, x) = \frac{k_{n-1}+1}{n} A_{k_{n-1}+1}(\{c_j b_j\}, x).
		\end{align}
		
		By the first part of the proof we observe that the sequence $\{A_n(\{c_jb_j, \cdot\})\}_{n \in \N}$ is is b.u.e.m at zero on $(L^{\Phi}_+, \norm{\cdot}_{\Phi})$.
		
		Let $K= \sup_{n \in \N} \frac{k_n}{n}$. It follows from Remark \ref{rem about lower density} that $0<K < \infty$. Let $\epsilon, \delta >0$. Let $\gamma >0$ be such that for all $x \in L^{\Phi}_+$ there exists $e \in \CP(M) $ such that 
		\begin{align*}
			\tau(e^{\perp}) < \epsilon ~~ \text{ and }~~ \sup_{n \in \N} \norm{e A_n(\{c_jb_j\}, x) e}_{\infty} <\frac{\delta}{K}.
		\end{align*}
		Consequently,
		\begin{align*}
			\sup_{n \in \N} \norm{e A_n^{\mathbf{k}}(\{ b_j\}, x) e}_{\infty} 
			& = \sup_{n \in \N} \frac{k_{n-1} +1}{n} \norm{e A_{k_{n-1}+1}(\{c_j b_j\}, x) e}_{\infty} \\
			& \leq K \sup_{n \in \N} \norm{e A_n(\{c_jb_j\}, x) e}_{\infty} \\
			& < K \frac{\delta}{K}= \delta.
		\end{align*}
		This completes the proof.	
	\end{proof}
	
	\begin{cor}\label{equi cont cor}
		Let $\{\beta_j\}_{j \in \N} \subset l^{\infty}(\C)$. If the strictly increasing sequence $\mathbf{k}:= \{k_j\}_{j \in \N}$ of natural numbers has lower density $d>0$, then the sequences $\{A_n^{\beta}\}_{n \in \N}$ and $\{A_n^{\beta, \mathbf{k}}\}_{n \in \N}$ are b.u.e.m. at zero on $(L^{\Phi}, \norm{\cdot}_{\Phi})$.
	\end{cor}
	
	\begin{rem}
		Let $\{\beta_j\}_{j \in \N} \subset l^{\infty}(\C)$. Note that it follows from \cite[Proposition 3.1]{OBrien:2021tw} that the sequences $\{A_n^{\beta}\}_{n \in \N}$ and $\{A_n^{\beta, \mathbf{k}}\}_{n \in \N}$ are b.u.e.m. at zero on $L^p$ for $(1 \leq p < \infty)$ where the sequence $\mathbf{k}:= \{k_j\}_{j \in \N}$ is of lower density $d>0$. Therefore, Corollary \ref{equi cont cor} substantially improves Proposition 3.1 of \cite{OBrien:2021tw}.
	\end{rem}
	
	As a consequence we prove the following proposition which is an important ingredient in proving our main result.
	\begin{prop}\label{bau-closed}
		Let $\{b_j\}_{j \in \N}$ be a bounded sequence in $\mathcal{Z}(M)$. If the strictly increasing sequence $\mathbf{k}:= \{k_j\}_{j \in \N}$ of natural numbers has lower density $d>0$, then the sets
		\begin{align*}
			& \s^{\{b_j\}}:= \set{x \in L^{\Phi} : \seq{A_n(\{b_j\},x)} \text{ converges } b.a.u} \text{ and, }\\
			& \s^{\{b_j\}, \mathbf{k}}:= \set{x \in L^{\Phi} : \seq{A^{\mathbf{k}}_n(\{b_j\},x)} \text{ converges } b.a.u}
		\end{align*}
		are closed in $L^{\Phi}$.
	\end{prop}
	
	\begin{proof}
		Since $(L^{\Phi}, \norm{\cdot}_{\Phi})$ is a Banach space and $\seq{A_n(\{b_j\}, \cdot)}$ and $\seq{A_n^{\mathbf{k}}(\{b_j\}, \cdot)}$ are sequences of additive maps, the result follows immediately from Theorem \ref{equi cont thm 1} and Theorem \ref{bau closed set 1}.
	\end{proof}
	
	\begin{rem}\label{bau closed set 2}
		Let $\beta:= \{\beta_j\}_{j \in \N} \subset l^{\infty}(\C)$ and $\mathbf{k}:= \{k_j\}_{j \in \N}$ be as stated in Proposition \ref{bau-closed}. Then we remark that it is evident from Proposition \ref{bau-closed} that the sets
		\begin{align*}
			& \s^{\beta}:= \set{x \in L^{\Phi} : \seq{A^{\beta}_n(x)} \text{ converges } b.a.u} \text{ and, }\\
			& \s^{\beta, \mathbf{k}}:= \set{x \in L^{\Phi} : \seq{A^{\beta, \mathbf{k}}_n(x)} \text{ converges } b.a.u}
		\end{align*}
		are closed in $L^{\Phi}$.
	\end{rem}
	
	In what follows $U(M)$ will always denote the group of unitary operators in $M$ and $\sigma(x)$ will denote the spectrum of an operator in $x \in M$. Let us define,
	\begin{align*}
		U_f:= \{ u \in U(M): \sigma(u) \text{ is finite} \}.
	\end{align*}
	
	\begin{defn}\label{trig poly}
		Let $U_0 \subseteq U(M)$. A function $\psi: \N \to M$ is called a trigonometric polynomial over $U_0$ if for some $m \in \N$ there exists $\{z_j\}_1^m \subset \C$ and $\{u_j\}_1^m \subset U_0$ such that 
		\begin{align*}
			\psi(k)= \sum_{j=1}^{m} z_j u_j^k ~~, k \in \N.
		\end{align*}
	\end{defn}
	
	\noindent For a trigonometric polynomial $\psi$ over $U_0$ as defined above, it is clear that $\norm{\psi} \leq \sum_{j=1}^{m} \abs{z_j}$.
	
	\begin{defn}\label{besicovitch seq}
		Let $U_0 \subseteq U(M)$. A sequence $\{b_j\} \subset M$ is called $U_0$-besicovitch if for all $\epsilon>0$ there exists a trigonometric polynomial $\psi$ over $U_0$ such that 
		\begin{align*}
			\limsup_{n \to \infty} \frac{1}{n} \sum_{j=0}^{n-1} \norm{b_j- \psi(j)}_\infty \leq \epsilon.
		\end{align*}
		A $U_0$-besicovitch sequence $\{b_j\}$ is called bounded if $\sup_{j \in \N} \norm{b_j}_\infty < \infty$.
	\end{defn}
	
	Now we recall the following result from \cite{comez_ergodic_2013} regarding the convergence of sequence of ergodic averages and immediately after that we extend it to the case of ergodic averages along a sequence of  density 1.
	
	\begin{thm}\label{au conv 1}
		Let $\{b_j\}$ and $\{d_j\}$ be $U_f$-besicovitch sequences such that at least one of which is bounded. Then the averages $A_n(\{b_j\}, \{d_j\}, x)$ converge a.u. for all $x \in L^1 \cap M$.
	\end{thm}
	
	\begin{proof}
		For proof we refer to \cite[Theorem 5.1]{comez_ergodic_2013}.
	\end{proof}
	
	\begin{thm}\label{conv in L1 cap M thm}
		Let $\{b_j\}$ and $\{d_j\}$ be $U_f$-besicovitch sequences with at least one of them is bounded and $\{k_j\}$ be a strictly increasing sequence of natural numbers of  density $1$. Then the sequence of averages $A_n^{\mathbf{k}}(\{b_j\}, \{d_j\}, x))$ converges a.u. for all $x \in L^1 \cap M$.
	\end{thm}
	
	\begin{proof}
		Without loss of generality we assume that $\{d_j\}$ is bounded and define $C:= \sup_j \norm{d_j} < \infty$. Fix $\epsilon>0$ and let $\psi_1(\cdot)= \sum_{i=1}^{m} z_i u_i^{(\cdot)}$ and $\psi_2(\cdot)= \sum_{i=1}^{l} w_i v_i^{(\cdot)}$ be such that $\{z_i\}, \{w_i\} \subset \C$, $\{u_i\}, \{v_i\} \subset U_f$ and
		\begin{align}\label{conv in L1 cap M eq 1}
			\limsup_{n \to \infty} \frac{1}{n} \sum_{j=0}^{n-1} \norm{b_j- \psi_1(j)}_\infty \leq \epsilon, ~~
			\limsup_{n \to \infty} \frac{1}{n} \sum_{j=0}^{n-1} \norm{d_j- \psi_2(j)}_\infty \leq \epsilon .
		\end{align}
		
		Let $x \in L^1 \cap M$. Note that by Theorem \ref{au conv 1} the averages $A_n(\{b_j\}, \{d_j\}, x)$ converges a.u. In particular, the averages $A_n(\{\psi_1(j)\}, \{\psi_2(j)\}, x)$ converges a.u. Hence the subsequence $A_{k_n}(\{\psi_1(j)\}, \{\psi_2(j)\}, x)$ converges a.u. Define,
		\begin{align*}
			M_n(\{\psi_1(j)\}, \{d_j\}, x):= \frac{1}{k_n} \sum_{j=0}^{n-1}  T^{k_j} (\psi_1(k_j) x d_{k_j}),~~ n \in \N.
		\end{align*}
		
		Now, we have
		\begin{align*}
			&
			\norm{A_{k_n}(\{\psi_1(j)\}, \{\psi_2(j)\}, x) - M_n(\{\psi_1(j)\}, \{d_j\}, x)} \\
			=& 
			\norm{\frac{1}{k_n} \sum_{j=0}^{k_n-1}  T^j (\psi_1(j) x \psi_2(j)) - \frac{1}{k_n} \sum_{j=0}^{n-1}  T^{k_j} (\psi_1(k_j) x d_{k_j})}\\
			\leq& 
			\norm{\frac{1}{k_n} \sum_{j=0}^{k_n-1}  T^j (\psi_1(j) x \psi_2(j)) - \frac{1}{k_n} \sum_{j=0}^{k_n-1}  T^j (\psi_1(j) x d_j)} \\
			& \qquad \qquad \qquad + \norm{\frac{1}{k_n} \sum_{j=0}^{k_n-1}  T^j (\psi_1(j) x d_j) - \frac{1}{k_n} \sum_{j=0}^{n-1}  T^{k_j} (\psi_1(k_j) x d_{k_j})} \\
			\leq&
			\frac{1}{k_n} \sum_{j=0}^{k_n-1} \norm{d_j - \psi_2(j)} \norm{x} \norm{\psi_1} + \norm{\frac{1}{k_n} \sum_{j=0, j \notin \mathbf{k}}^{k_n-1} T^j (\psi_1(j) x d_j) } ~~(\text{ since $\norm{T} <1$})\\
			\leq&
			\norm{\psi_1} \norm{x} \epsilon + \frac{1}{k_n} \sum_{j=0, j \notin \mathbf{k}}^{k_n-1} \norm{\psi_1} C \norm{x} ~~ (\text{ since $\norm{T}< 1$ and by Eq. \ref{conv in L1 cap M eq 1}})  \\
			\leq& 
			\norm{\psi_1} \norm{x} \epsilon + \frac{k_n-n}{k_n} \norm{\psi_1} C \norm{x}.
		\end{align*}
		
		Now since $\frac{k_n-n}{k_n} \to 0$ as $n \to \infty$, we can choose $N \in \N$ such that for all $n \geq N$ we have
		\begin{align*}
			\norm{A_{k_n}(\{\psi_1(j)\}, \{\psi_2(j)\}, x) - M_n(\{\psi_1(j)\}, \{d_j\}, x)} < \epsilon.
		\end{align*}
		
		Hence, it follows from Lemma \ref{app by covg seq} that the sequence $\{M_n(\{\psi_1(j)\}, \{d_j\}, x)\}$ converges a.u. Again, define
		\begin{align*}
			M_n(\{b_j\}, \{d_j\}, x):= \frac{1}{k_n} \sum_{j=0}^{n-1}  T^{k_j} (b_{k_j} x d_{k_j}),~~ n \in \N.
		\end{align*}
		
		Then,
		\begin{align*}
			&\norm{M_n(\{b_j\}, \{d_j\}, x) - M_n(\{\psi_1(j)\}, \{d_j\}, x)}  \\
			&= \norm{\frac{1}{k_n} \sum_{j=0}^{n-1}  T^{k_j} (b_{k_j} x d_{k_j}) - \frac{1}{k_n} \sum_{j=0}^{n-1}  T^{k_j} (\psi_1(k_j) x d_{k_j})} \\
			& \leq \frac{1}{k_n} \sum_{j=0}^{n-1} \norm{b_{k_j} - \psi_1(k_j)} \norm{x} C ~~(\text{ since $\norm{T} <1$})\\
			& \leq \frac{1}{k_n} \sum_{j=0}^{k_n-1} \norm{b_j - \psi_1(j)} \norm{x} C \\
			& \leq \epsilon \norm{x} C (\text{ by Eq. \ref{conv in L1 cap M eq 1}}).
		\end{align*}
		Hence, an appeal to Lemma \ref{app by covg seq} implies that the sequence $\{M_n(\{b_j\}, \{d_j\}, x)\}$ converges a.u. Now since $\lim_{n \to \infty} \frac{k_n}{n} = 1$ and $A_n^{\mathbf{k}}(\{b_j\}, \{d_j\}, x)= \frac{k_n}{n} M_n(\{b_j\}, \{d_j\}, x)$ for all $n \in \N$, the result follows immediately.
	\end{proof}
	
	\begin{cor}\label{conv in L1 cap M cor}
		Let $\{b_j\}_{j \in \N}$ be a $U_f$-besicovitch sequence and $\{k_j\}_{j \in \N}$ has  density $1$. Let $x \in L^1 \cap M$. Then the averages 
		\begin{align*}
			\frac{1}{n} \sum_{j=0}^{n-1}  T^{k_j} (b_{k_j} x) ~~ \text{ and } ~~ \frac{1}{n} \sum_{j=0}^{n-1}  T^{k_j} (x b_{k_j})
		\end{align*}
		converges a.u.
	\end{cor}
	
	As a consequence we obtain the individual ergodic theorem for vector valued Besicovitch weight along a sequence of density one.
	
	\begin{thm}\label{main thm}
		Assume that the Orlicz function $\Phi$ satisfies $\Delta_2$ condition. Let $\mathbf{k}:=\seq{k_j}$ be a sequence of density $1$ and $\{b_j\}_{j \in \N}$ be a bounded $U_f$-besicovitch sequence in $\mathcal{Z}(M)$. Then for every $x \in L^{\Phi}$ the sequence $\{A_n^{\mathbf{k}}(\{b_j\}, x)\}$ converges b.a.u to some $\hat{x} \in L^{\Phi}$.
	\end{thm}
	
	\begin{proof}
		Define, $\s^{\{b_j\}, \mathbf{k}}:= \set{x \in L^{\Phi} : \seq{A^{\mathbf{k}}_n(\{b_j\},x)} \text{ converges } b.a.u}$. Note that, by Proposition \ref{bau-closed} the set $\s^{\{b_j\}, \mathbf{k}}$ is closed in $L^{\Phi}$. Since $L^1 \cap M$ is dense in $L^{\Phi}$, we have $\s^{\{b_j\}, \mathbf{k}}= L^{\Phi} $.
		
		Let $x \in L^{\Phi}$. Then by Proposition \ref{orlicz norm prop} $\{A^{\mathbf{k}}_n(\{b_j\},x)\}_{n \in \N} \subset L^\Phi$. Also there exists $\hat{x} \in L^0$ such that $A^{\mathbf{k}}_n(\{b_j\},x)$ converges b.a.u. to $\hat{x}$, hence in measure. Now since $\norm{T}\leq 1$, we observe that for all $n \in \N$,
		\begin{align*}
			\norm{A^{\mathbf{k}}_n(\{b_j\},x)}_{\Phi} \leq \frac{1}{n} \sum_{j=0}^{n-1} \norm{b_{k_j}} \norm{ x}_{\Phi} \leq C \norm{x}_{\Phi},
		\end{align*}
		where $C= \sup_{j \in \N} \norm{b_j} < \infty$.
		Therefore, for all $n \in \N$, $A^{\mathbf{k}}_n(\{b_j\},x)$ belongs to the closed ball of $(L^{\Phi}, \norm{\cdot}_{\Phi})$ of radius $C \norm{x}_{\Phi}$. Consequently by Remark \ref{unit ball closed}, $\hat{x} \in L^{\Phi}$.
	\end{proof}
	
	\begin{rem}
		\begin{enumerate}
			\item Following Definition \ref{trig poly} and \ref{besicovitch seq} one can always define a scalar valued Besicovitch sequence. In particular, A scalar valued trigonometric polynomial is a function $P: \N \to \C$ satisfying 
			\begin{align*}
				P(k)= \sum_{j=1}^s r_j \lambda_j^k, ~~~ k \in \Z
			\end{align*} 
			for some $\{r_j\}_{j=1}^s \subset \C$ and $\{\lambda_j\}_{j=1}^s \subset \C^1$, where $\C^1:= \{z \in \C: \abs{z}=1 \}$. A sequence $\seq{\beta_j}_{j=1}^{\infty}$ of complex numbers is called a Besicovitch sequence if for all $\epsilon > 0$ there exists a trigonometric polynomial $P$ such that
			\begin{align*}
				\limsup_{n \to \infty} \frac{1}{n} \sum_{j=0}^{n-1} \abs{\beta_j- P(j)} < \epsilon.
			\end{align*} 
			The sequence $\seq{\beta_j}_{j=1}^{\infty}$ is bounded is $\sup_{j \in \N} \abs{\beta_j}< \infty$.
			\item Very recently in \cite[Corollary 3.2]{OBrien:2021tw}, the author proved the conclusion of Theorem \ref{main thm} when $x \in L^p~ (1\leq p< \infty)$ and also under the hypothesis that the Besicovitch weights are scalar valued. Hence our theorem generalises Corollary 3.2 of \cite{OBrien:2021tw}.
		\end{enumerate}
	\end{rem}


\begin{thebibliography}{AAB{\etalchar{+}}10}
	
	\bibitem[AAB{\etalchar{+}}10]{Anantharaman:2010ts}
	Claire Anantharaman, Jean-Philippe Anker, Martine Babillot, Aline Bonami, Bruno
	Demange, Sandrine Grellier, Fran\c{c}ois Havard, Philippe Jaming, Emmanuel
	Lesigne, Patrick Maheux, Jean-Pierre Otal, Barbara Schapira, and Jean-Pierre
	Schreiber, \emph{Th\'{e}or\`emes ergodiques pour les actions de groupes},
	Monographies de L'Enseignement Math\'{e}matique [Monographs of L'Enseignement
	Math\'{e}matique], vol.~41, L'Enseignement Math\'{e}matique, Geneva, 2010,
	With a foreword in English by Amos Nevo. \MR{2643350}
	
	\bibitem[BS22]{Panchugopal:2022ub}
	P. Bikram and D. Saha, \emph{On the non-commuttative Neveu decomposition and stochastic ergodic theorems},  J. Funct. Anal., 284 (2023), no 1,  Paper No. 109706, 45 pp.  
	

	
	\bibitem[CDN78]{Conze1978}
	J.-P. Conze and N.~Dang-Ngoc, \emph{Ergodic theorems for noncommutative
		dynamical systems}, Inventiones Mathematicae \textbf{46} (1978), no.~1,
	1--15. \MR{500185}
	
	\bibitem[cL13]{comez_ergodic_2013}
	Do\u{g}an \c{C}\"{o}mez and Semyon Litvinov, \emph{Ergodic averages with
		vector-valued {B}esicovitch weights}, Positivity \textbf{17} (2013), no.~1,
	27--46. \MR{3027644}
	
	\bibitem[CL15]{Chilin:2015wo}
	Vladimir Chilin and Semyon Litvinov, \emph{Ergodic theorems in fully symmetric
		spaces of {$\tau$}-measurable operators}, Studia Math. \textbf{228} (2015),
	no.~2, 177--195. \MR{3422091}
	
	\bibitem[CL17]{Chilin:2017uj}
	\bysame, \emph{Individual ergodic theorems in noncommutative {O}rlicz spaces},
	Positivity \textbf{21} (2017), no.~1, 49--59. \MR{3612983}
	
	\bibitem[CLS05]{chilin_few_2005}
	Vladimir Chilin, Semyon Litvinov, and Adam Skalski, \emph{A few remarks in
		non-commutative ergodic theory}, Journal of Operator Theory \textbf{53}
	(2005), no.~2, 331--350. \MR{2153152}
	
	\bibitem[DDST05]{Dodds:2005ty}
	P.~G. Dodds, T.~K. Dodds, F.~A. Sukochev, and O.~Ye. Tikhonov, \emph{A
		non-commutative {Y}osida-{H}ewitt theorem and convex sets of measurable
		operators closed locally in measure}, Positivity \textbf{9} (2005), no.~3,
	457--484. \MR{2188531}
	
	\bibitem[FK86]{Fack:1986ua}
	Thierry Fack and Hideki Kosaki, \emph{Generalized {$s$}-numbers of
		{$\tau$}-measurable operators}, Pacific J. Math. \textbf{123} (1986), no.~2,
	269--300. \MR{840845}
	
	\bibitem[Hia21]{Hiai:2021vl}
	Fumio Hiai, \emph{Lectures on selected topics in von {N}eumann algebras}, EMS
	Series of Lectures in Mathematics, EMS Press, Berlin, [2021] \copyright 2021.
	\MR{4331436}
	
	\bibitem[HLW21]{Hong:2021vm}
	Guixiang Hong, Ben Liao, and Simeng Wang, \emph{Noncommutative maximal ergodic
		inequalities associated with doubling conditions}, Duke Math. J. \textbf{170}
	(2021), no.~2, 205--246. \MR{4202493}
	
	\bibitem[Hon20]{Hong:2020wq}
	Guixiang Hong, \emph{Non-commutative ergodic averages of balls and spheres over
		{E}uclidean spaces}, Ergodic Theory Dynam. Systems \textbf{40} (2020), no.~2,
	418--436. \MR{4048299}
	
	\bibitem[Hu08]{Ying:2008}
	Ying Hu, \emph{Maximal ergodic theorems for some group actions}, J. Funct.
	Anal. \textbf{254} (2008), no.~5, 1282--1306. \MR{2386939}
	
	\bibitem[Jaj85]{Jajte:1985vd}
	Ryszard Jajte, \emph{Strong limit theorems in noncommutative probability},
	Lecture Notes in Mathematics, vol. 1110, Springer-Verlag, Berlin, 1985.
	\MR{778724}
	
	\bibitem[JX07]{Junge2007}
	Marius Junge and Quanhua Xu, \emph{Noncommutative maximal ergodic theorems},
	Journal of the American Mathematical Society \textbf{20} (2007), no.~2,
	385--439. \MR{2276775}
	
	\bibitem[K\"78]{Kuemmerer1978}
	Burkhard K\"{u}mmerer, \emph{A non-commutative individual ergodic theorem},
	Inventiones Mathematicae \textbf{46} (1978), no.~2, 139--145. \MR{482260}
	
	\bibitem[Kun90]{Kunze:1990vc}
	Wolfgang Kunze, \emph{Noncommutative {O}rlicz spaces and generalized {A}rens
		algebras}, Math. Nachr. \textbf{147} (1990), 123--138. \MR{1127316}
	
	\bibitem[Lan76]{Lance1976}
	E.~Christopher Lance, \emph{Ergodic theorems for convex sets and operator
		algebras}, Inventiones Mathematicae \textbf{37} (1976), no.~3, 201--214.
	\MR{428060}
	
	\bibitem[Lit12]{Litvinov:2012wk}
	Semyon Litvinov, \emph{Uniform equicontinuity of sequences of measurable
		operators and non-commutative ergodic theorems}, Proc. Amer. Math. Soc.
	\textbf{140} (2012), no.~7, 2401--2409. \MR{2898702}
	
	\bibitem[LM01]{Litvinov:2001wu}
	Semyon Litvinov and Farrukh Mukhamedov, \emph{On individual subsequential
		ergodic theorem in von {N}eumann algebras}, Studia Math. \textbf{145} (2001),
	no.~1, 55--62. \MR{1828992}
	
	\bibitem[O'B21]{OBrien:2021tw}
	Morgan O'Brien, \emph{Some noncommutative subsequential weighted individual
		ergodic theorems}, Infin. Dimens. Anal. Quantum Probab. Relat. Top.
	\textbf{24} (2021), no.~3, Paper No. 2150018, 17. \MR{4348153}
	
	\bibitem[PX03]{Pisier:2003wz}
	Gilles Pisier and Quanhua Xu, \emph{Non-commutative {$L^p$}-spaces}, Handbook
	of the geometry of {B}anach spaces, {V}ol. 2, North-Holland, Amsterdam, 2003,
	pp.~1459--1517. \MR{1999201}
	
	\bibitem[Ros94]{Rosenblatt:1994vk}
	Joseph Rosenblatt, \emph{Norm convergence in ergodic theory and the behavior of
		{F}ourier transforms}, Canad. J. Math. \textbf{46} (1994), no.~1, 184--199.
	\MR{1260343}
	
	\bibitem[Yea77]{Yeadon1977}
	F.~J. Yeadon, \emph{Ergodic theorems for semifinite von {N}eumann algebras.
		{I}}, Journal of the London Mathematical Society. Second Series \textbf{16}
	(1977), no.~2, 326--332. \MR{487482}
	
\end{thebibliography}

\newcommand{\etalchar}[1]{$^{#1}$}
\providecommand{\bysame}{\leavevmode\hbox to3em{\hrulefill}\thinspace}
\providecommand{\MR}{\relax\ifhmode\unskip\space\fi MR }
\providecommand{\MRhref}[2]{%
	\href{http://www.ams.org/mathscinet-getitem?mr=#1}{#2}
}
\providecommand{\href}[2]{#2}

\end{document}